\documentclass[11pt,pdflatex,sn-mathphys-num]{sn-jnl}

\providecommand{\texorpdfstring}[2]{#1}
\makeatletter
\@ifpackageloaded{hyperref}{%
  \pdfstringdefDisableCommands{\def\eqref#1{(\ref{#1})}}%
}{}
\makeatother

\usepackage{graphicx}%
\usepackage{multirow}%
\usepackage{amsmath,amssymb,amsfonts}%
\usepackage{amsthm}%
\usepackage{mathrsfs}%
\usepackage[title]{appendix}%
\usepackage{xcolor}%
\usepackage{textcomp}%
\usepackage{manyfoot}%
\usepackage{booktabs}%
\usepackage{algorithm}%
\usepackage{algorithmicx}%
\usepackage{algpseudocode}%
\usepackage{listings}%

\theoremstyle{thmstyleone}%
\newtheorem{theorem}{Theorem}
\theoremstyle{thmstyletwo}%
\newtheorem{remark}{Remark}%

\theoremstyle{thmstylethree}%
\newtheorem{definition}{Definition}%

\begin{document}

\title[Article Title]{Explicit representation of solutions to a linear wave equation with time delay}


\author[1]{\fnm{Javad} \sur{Asadzade}}\email{javad.asadzade@emu.edu.tr}
\author[1]{\fnm{Jasarat} \sur{Gasimov}}\email{jasarat.gasimov@emu.edu.tr}
\author[1]{\fnm{Nazim} \sur{Mahmudov}}\email{nazim.mahmudov@emu.edu.tr}
\author*[2,3]{\fnm{Ismail} \sur{Huseynov}}\email{ismail.huseynov@ptb.de}

 \affil[1]{Eastern Mediterranean University, Mersin 10, 99628, Turkey}
  \affil[2]{Physikalisch-Technische Bundesanstalt, Berlin 10587, Germany}
  \affil[3]{Technical University of Berlin, Berlin 10623, Germany}


\abstract{This paper develops an explicit spectral representation for solutions of a one-dimensional linear wave equation with a constant time delay. The model is considered on a bounded interval with non-homogeneous Dirichlet boundary data and a prescribed history function. To accommodate the loss of global smoothness in time caused by delay terms, solutions are understood in a \textit{stepwise classical sense}, allowing jump discontinuities in the second time derivative at multiples of the delay while maintaining continuity of the solution and its first time derivative. By combining separation of variables with Sturm-Liouville expansions, the delayed PDE is reduced to a family of scalar second-order delay differential equations. Using delay-dependent fundamental solutions, we derive closed-form representation formulas for the modal dynamics and reconstruct the PDE solution as a Fourier series. Convergence conditions guaranteeing uniform convergence and admissibility of termwise differentiation in space are established. A numerical example demonstrates the practical computation of truncated series solutions and their visualization.}

\keywords{Delayed wave equation, stepwise classical solution, separation of variables, Sturm-Liouville problem, delay differential equations, delayed perturbation functions, Fourier series convergence.}



\maketitle

\section{Introduction}\label{Sec:intro}

The propagation of waves in media with delayed responses has attracted considerable research interest due to its relevance in physics, engineering, and control theory. Delay effects naturally arise in systems where the present state depends not only on the current configuration but also on past behavior. Typical examples include viscoelastic materials, memory-type constitutive laws, and feedback-controlled vibrating structures. The inclusion of time-delay terms in hyperbolic models can substantially affect qualitative properties of solutions, including stability, controllability, and energy distribution.

Classical wave equations without delay have been extensively studied. For instance, Tikhonov and Samarskii~\cite{Tikhonov} provide foundational results on well-posedness and solution methods for second-order partial differential equations; Travis and Webb~\cite{Travis} and Lutz~\cite{Lutz} investigate bounded perturbations and abstract operator cosine functions; and Lin~\cite{Lin} develops time-dependent perturbation theory for abstract second-order evolution equations.

The presence of delay terms significantly increases analytical complexity. Nicaise and Pignotti~\cite{Nicaise} and Nicaise and Valein~\cite{Nicaise2} studied stability and instability of wave equations with boundary or internal delays. Kirane and Said-Houari~\cite{Kirane} analyzed existence and asymptotic stability for viscoelastic wave models with delays. Abuasbeh et al.~\cite{Abuasbeh} investigated controllability and Ulam--Hyers stability for second-order linear time-delay systems, while Assanova et al.~\cite{Assanova} addressed well-posedness of periodic hyperbolic systems with finite delays. Further contributions include the long-time behavior studied by Liu et al.~\cite{Liu} and exponential stabilization for input-delayed one-dimensional wave equations developed by Wang and Xu~\cite{Wang}.

A key direction in the literature is the construction of explicit solution representations. Khusainov, Pokojovy, and Azizbayov~\cite{Khusainov, Azizbayov} obtained explicit representations for one-dimensional wave equations in the presence of a pure time-delay mechanism by introducing delay-adapted sine/cosine-type functions. Subsequent works~\cite{Khusainov2, Khusainov3, Khusainov4, Khusainov5, Huseynov, Huseynov2} extended related techniques to heat-type equations, fractional models, thermoelasticity, harmonic oscillators, and exact control constructions. The closest explicit-series works to the present setting are Khusainov, Pokojovy and Azizbayov~\cite{Khusainov}, Rodríguez et al.~\cite{Rodríguez}, and the one-dimensional case of Jornet~\cite{Jornet}; our equation keeps both present and delayed spatial operators, including first-order spatial terms, together with non-homogeneous boundary data and forcing. Lobo and Valaulikar~\cite{Lobo} applied group-theoretical methods to one-dimensional wave equations with delay. These studies demonstrate that explicit representations are valuable both for qualitative analysis and for efficient numerical evaluation.

In this paper, we consider the following generalized one-dimensional linear wave equation with a single constant \emph{time delay}:
\begin{equation}\label{e.1}
\begin{cases}
\partial_{tt}u(t,x)= a_{1}^{2}\partial_{xx}u(t,x)+a_{2}^{2}\partial_{xx}u(t-\tau,x)
+b_{1}\partial_x u(t,x)+b_{2}\partial_x u(t-\tau,x)\\
\qquad\qquad + \;\ c_{1}u(t,x)+c_{2}u(t-\tau,x)+g(t,x), \quad (t,x)\in[0,T]\times[0,L],\\[1mm]
u(t,0)=\theta_{1}(t),\quad u(t,L)=\theta_{2}(t), \quad t\in[-\tau,T],\\
u(t,x)=\psi(t,x), \quad (t,x)\in[-\tau,0]\times[0,L].
\end{cases}
\end{equation}
where $\tau>0$ is fixed and $a_i,b_i,c_i$ $(i=1,2)$ are real coefficients. The data $g$, $\psi$, $\theta_1$, and $\theta_2$ are assumed to be sufficiently smooth and compatible so that the solution representation and the convergence results below are meaningful.

The main objective is to construct an explicit analytical representation of the solution of~\eqref{e.1} under non-homogeneous boundary conditions and prescribed history. A fundamental point is that time-delay terms typically prevent global $C^{2}$-regularity in time: even for smooth history data, the second time derivative may develop jump discontinuities at the delay grid points $t=k\tau$. For this reason we work with a \emph{stepwise (piecewise-in-time) classical} notion of solution: $u$ and $\partial_t u$ remain continuous, while $\partial_{tt}u$ is continuous on each open step interval and may jump at $t=k\tau$.

Our approach is based on separation of variables and Sturm--Liouville theory. After an analytic simplification step (including an exponential transformation that removes first-order spatial derivatives under a natural coefficient compatibility), the PDE is reduced to a countable family of second-order linear delay differential equations for the Fourier modes. For these scalar delay equations we employ the fundamental solutions $C_{\tau}^{a,b}$ and $S_{\tau}^{a,b}$ (delayed perturbation functions), which generalize the classical cosine/sine-type fundamental solutions for the corresponding undelayed equation (e.g.\ when the delayed coupling vanishes). This yields explicit representation formulas for both the homogeneous and the forced modal dynamics and hence for the PDE solution as a Fourier series.

The present framework extends the explicit-representation philosophy of Khusainov et al.~\cite{Khusainov} to a broader class of delayed wave models that include both instantaneous and delayed spatial operators. In particular, the pure-delay model is recovered as a special case by setting the instantaneous coefficients to zero, i.e.\ $a_{1}=b_{1}=c_{1}=0$. More precisely, this special-case statement is understood at the level of the reduced one-dimensional equation; when the exponential change of variables below involves division by $a_i^2$, degenerate cases with $a_i=0$ should be obtained directly from the already simplified equation rather than from the transformation formula itself. We also provide sufficient conditions ensuring absolute and uniform convergence of the resulting Fourier series and of the termwise differentiated series in space, thereby justifying the formal computations and validating numerical evaluation via truncated expansions. Finally, a numerical example illustrates the practical computation and visualization of truncated series solutions.

The paper is organized as follows. Section~\ref{sec3} introduces the analytical simplification and the stepwise classical solution framework. Section~\ref{sec4} treats the homogeneous problem and derives the modal solution formulas, while Section~\ref{sec5} addresses the non-homogeneous case. Section~\ref{sec7} assembles the general solution representation. Section~\ref{sec8} establishes convergence and termwise differentiation results. Section~\ref{sec9} presents a numerical example and concludes with a discussion of the implications of our findings.

Throughout this paper, we assume that the end time $T$ is a multiple of the delay $\tau>0$, i.e.\ $T=N\tau$ for some $N\in\mathbb{N}$.
To shorten notation, we define the open, left-open, and closed step intervals by
\[
I_k^\circ:=((k-1)\tau,k\tau),\qquad I_k:=((k-1)\tau,k\tau], \qquad \bar I_k:=[(k-1)\tau,k\tau],
\qquad k=1,\dots,N.
\]

\section{Analytical simplification of the wave equation}\label{sec3}

For $T>0$ and $L>0$, we consider the following linear wave equation with a single time-delay $\tau>0$ on the bounded interval $(0,L)$.
The unknown $u$ is governed, for $(t,x)\in[0,T]\times[0,L]$, by
\begin{align}\label{q.1}
\partial_{tt}u(t,x)
&=a^{2}_{1}\partial_{xx}u(t,x)+a^{2}_{2}\partial_{xx}u(t-\tau,x)+b_{1}\partial_{x}u(t,x)\nonumber\\
&+b_{2}\partial_{x}u(t-\tau,x)+c_{1}u(t,x)+c_{2}u(t-\tau,x)+g(t,x),
\end{align}
subject to non-homogeneous Dirichlet boundary and history conditions:
\begin{align}\label{q.2}
\begin{cases}
u(t,0)=\theta_{1}(t),\quad u(t,L)=\theta_{2}(t),\quad t\in [-\tau,T],\\
u(t,x)=\psi(t,x),\quad (t,x)\in [-\tau,0]\times [0,L].
\end{cases}
\end{align}

\begin{definition}[Stepwise (piecewise-in-time) classical solution]\label{def:stepwise-classical}
A \emph{stepwise classical solution} to \eqref{q.1}--\eqref{q.2} is a function
$u:[-\tau,T]\times[0,L]\to\mathbb{R}$ such that:
\begin{enumerate}[{\rm(i)}]
\item $u\in C([-\tau,T]\times[0,L])$ and $u(t,x)=\psi(t,x)$ for $(t,x)\in[-\tau,0]\times[0,L]$.

\item The spatial derivatives $\partial_x u$ and $\partial_{xx}u$ exist on $[-\tau,T]\times[0,L]$ and
\[
\partial_x u,\;\partial_{xx}u\in C([-\tau,T]\times[0,L]).
\]

\item The time derivative $\partial_t u$ and mixed derivative $\partial_{tx}u$ exist and are continuous on $[-\tau,T]\times[0,L]$.
\item For every $k=1,\dots,N$, the second time derivative $\partial_{tt}u$ exists and is continuous on $I_k^\circ\times[0,L]$,
and $\partial_{tt}u$ admits a continuous extension to $\bar I_k\times[0,L]$. In other words, it has finite one-sided limits at the step endpoints
(from within the interval $I_k^\circ$), so it may jump at $t=k\tau$.

\item For all $t\in[-\tau,T]$,
\[
u(t,0)=\theta_1(t),\qquad u(t,L)=\theta_2(t),
\]
and the data satisfy the compatibility on $[-\tau,0]$:
\[
\psi(t,0)=\theta_1(t),\qquad \psi(t,L)=\theta_2(t)\qquad \text{for }t\in[-\tau,0].
\]

\item For each $k=1,\dots,N$ the identity \eqref{q.1} holds pointwise for all $(t,x)\in I_k^\circ\times[0,L]$
(with delayed terms evaluated at $t-\tau$).
\end{enumerate}
\end{definition}

Next, to reduce the number of partial derivative terms, we apply the linear transformation:
\begin{align}\label{q.3}
u(t,x)=e^{-\alpha x}v(t,x),\quad \text{for}\quad (t,x)\in [-\tau,T]\times [0,L],
\end{align}
where $v(t,x)$ is a new unknown function. Differentiating \eqref{q.3} yields
\begin{align}\label{q.4}
\begin{cases}
\partial_{tt}u(t,x)=e^{-\alpha x}\partial_{tt}v(t,x),\\
\partial_{x}u(t,x)=e^{-\alpha x}\big(-\alpha v(t,x) +\partial_{x}v(t,x) \big),\\
\partial_{xx}u(t,x)=e^{-\alpha x}\big(\alpha ^{2} v(t,x) -2\alpha \partial_{x}v(t,x) +\partial_{xx}v(t,x) \big).
\end{cases}
\end{align}
Substituting these expressions into \eqref{q.1}, we obtain the equation for $v$:
\begin{align*}
\partial_{tt}v(t,x)
&=\Big[a_{1}^{2}\partial_{xx}v(t,x)+(b_{1}-2\alpha a_{1}^{2}) \partial_{x}v(t,x)
+(c_{1}+\alpha^{2}a_{1}^{2}-\alpha b_{1})v(t,x)\Big]+e^{\alpha x}g(t,x)\\
&\quad+\Big[a_{2}^{2}\partial_{xx}v(t-\tau,x)+(b_{2}-2\alpha a_{2}^{2}) \partial_{x}v(t-\tau,x)
+(c_{2}+\alpha^{2}a_{2}^{2}-\alpha b_{2})v(t-\tau,x)\Big].
\end{align*}
To simplify the equation, we impose
\[
b_{1}-2\alpha a_{1}^{2}=0,\qquad b_{2}-2\alpha a_{2}^{2}=0,
\]
which implies
\[
\alpha=\frac{b_{1}}{2a_{1}^{2}}\qquad \text{and}\qquad \alpha=\frac{b_{2}}{2a_{2}^{2}}.
\]
Hence, we require the coefficient compatibility condition
\begin{align}\label{q.5}
b_{1}a_{2}^{2}=b_{2}a_{1}^{2},
\end{align}
which is equivalent to $b_{1}/b_{2}=a_{1}^{2}/a_{2}^{2}$ whenever the displayed ratios are well defined.
Choosing $\alpha=\frac{b_{1}}{2a_{1}^{2}}$ (equivalently $\alpha=\frac{b_2}{2a_2^2}$ under \eqref{q.5}), we arrive at the reduced equation
\begin{align}\label{q.6}
\partial_{tt}v(t,x)
=a^{2}_{1}\partial_{xx}v(t,x)+a^{2}_{2}\partial_{xx}v(t-\tau,x)+d_{1}v(t,x)+d_{2}v(t-\tau,x)+f(t,x),
\end{align}
where
\begin{align}\label{q.7}
d_{1}=c_{1}-\frac{b^{2}_{1}}{4a_{1}^{2}},\qquad
d_{2}=c_{2}-\frac{b^{2}_{2}}{4a_{2}^{2}},\qquad
f(t,x)=e^{\frac{b_{1}}{2a_{1}^{2}}x}g(t,x).
\end{align}
Thus, the original problem is reduced to the following problem for $v$ on $(t,x)\in[0,T]\times[0,L]$:
\begin{align}\label{q.8}
\begin{cases}
\partial_{tt}v(t,x)=a^{2}_{1}\partial_{xx}v(t,x)+a^{2}_{2}\partial_{xx}v(t-\tau,x)+d_{1}v(t,x)+d_{2}v(t-\tau,x)+f(t,x),\\[1mm]
v(t,0)=\mu_{1}(t),\quad v(t,L)=\mu_{2}(t),\quad t\in [-\tau,T],\\[1mm]
v(t,x)=\varphi(t,x),\quad (t,x)\in[-\tau,0]\times[0,L],
\end{cases}
\end{align}
with
\[
\mu_{1}(t)=\theta_{1}(t),\qquad
\mu_{2}(t)=e^{\frac{b_{1}}{2a_{1}^{2}}L}\theta_{2}(t),\qquad
\varphi(t,x)=e^{\frac{b_{1}}{2a_{1}^{2}}x}\psi(t,x).
\]

\begin{remark}
Under the compatibility condition \eqref{q.5}, the two choices
$\alpha=\frac{b_{1}}{2a_{1}^{2}}$ and $\alpha=\frac{b_{2}}{2a_{2}^{2}}$ coincide; hence either choice yields the same reduced form.
If \eqref{q.5} fails, one may still choose $\alpha$ to eliminate the non-delayed first-order term, but the delayed first-order term cannot be removed simultaneously.
If one of the second-order coefficients is zero, the corresponding first-order coefficient must be treated separately; in particular, the formulas for $d_i$ in \eqref{q.7} apply directly under the non-degenerate assumption $a_1a_2\ne0$.
\end{remark}

We determine the solution as a linear superposition of three terms:
\begin{align}\label{q.9}
v(t,x)=v_{0}(t,x)+v_{1}(t,x)+G(t,x).
\end{align}
Here $G$ is chosen to satisfy the boundary conditions:
assume $\mu_1,\mu_2\in C^{2}([-\tau,T])$ and define
\begin{align}\label{q10}
G(t,x):=\mu_1(t)+\frac{x}{L}\big(\mu_2(t)-\mu_1(t)\big),
\qquad (t,x)\in[-\tau,T]\times[0,L].
\end{align}
Then $G(t,0)=\mu_1(t)$, $G(t,L)=\mu_2(t)$ and $G_{xx}\equiv 0$.

\medskip
\noindent$\bullet$ The function $v_0$ solves the homogeneous delayed wave equation
\begin{align}\label{q.11}
\partial_{tt}v_0(t,x)
= a^2_{1}\partial_{xx}v_0(t,x)+a^2_{2}\partial_{xx}v_0(t-\tau,x)
+d_{1}v_0(t,x)+d_{2}v_0(t-\tau,x),
\end{align}
for $(t,x)\in[0,T]\times[0,L]$, subject to homogeneous boundary conditions
\[
v_0(t,0)=0,\qquad v_0(t,L)=0,\qquad t\in[-\tau,T],
\]
and history data
\[
v_0(t,x)=\Phi(t,x)\qquad \text{for }(t,x)\in[-\tau,0]\times[0,L],
\]
where
\[
\Phi(t,x):=\varphi(t,x)-G(t,x)
=\varphi(t,x)-\mu_1(t)-\frac{x}{L}\big(\mu_2(t)-\mu_1(t)\big),
\qquad (t,x)\in[-\tau,0]\times[0,L].
\]

\medskip
\noindent$\bullet$ The function $v_1$ solves the non-homogeneous delayed wave equation
\begin{align}\label{qq.12}
\partial_{tt}v_1(t,x)
= a^{2}_{1}\partial_{xx}v_1(t,x)+a^{2}_{2}\partial_{xx}v_1(t-\tau,x)
+d_{1}v_1(t,x)+d_{2}v_1(t-\tau,x)+F(t,x),
\end{align}
for $(t,x)\in[0,T]\times[0,L]$, subject to homogeneous boundary conditions
\[
v_1(t,0)=0,\qquad v_1(t,L)=0,\qquad t\in[-\tau,T],
\]
and (for definiteness) homogeneous history
\[
v_1(t,x)=0\qquad \text{for }(t,x)\in[-\tau,0]\times[0,L].
\]
The source term $F$ is obtained by inserting $v=v_0+v_1+G$ into \eqref{q.6} and using that $G_{xx}\equiv 0$, namely
\begin{align*}
F(t,x)
:= f(t,x)+d_{1}G(t,x)+d_{2}G(t-\tau,x)-\partial_{tt}G(t,x),
\qquad (t,x)\in[0,T]\times[0,L].
\end{align*}
Using \eqref{q10}, we obtain the explicit representation
\begin{align}\label{uu1}
F(t,x)
&= f(t,x)
+ d_{1}\left( \mu_1(t) + \frac{x}{L}\big(\mu_2(t) - \mu_1(t)\big) \right)\nonumber\\
&+ d_{2}\left( \mu_1(t - \tau) + \frac{x}{L}\big(\mu_2(t - \tau) - \mu_1(t - \tau)\big) \right) - \ddot{\mu}_1(t) - \frac{x}{L}\big(\ddot{\mu}_2(t)-\ddot{\mu}_1(t)\big).
\end{align}

This completes the decomposition of \eqref{q.8} into a homogeneous subproblem for $v_0$ and a forced subproblem for $v_1$,
both with homogeneous boundary conditions.

\section{Homogeneous equation}\label{sec4}

In this section, we derive a formal solution for the initial-boundary value problem \eqref{q.11} subject to the homogeneous boundary
conditions $v_0(t,0)=v_0(t,L)=0$ for $t\in[-\tau,T]$ and the history condition $v_0(t,x)=\Phi(t,x)$ for $(t,x)\in[-\tau,0]\times[0,L]$,
where $\Phi=\varphi-G$ is given in Section~\ref{sec3}. To this end, we apply the method of separation of variables and seek a solution of the form
\[
v_0(t,x)=T(t)X(x).
\]
Substituting this ansatz into \eqref{q.11} yields
\begin{align*}
X(x)\ddot{T}(t)
&=a_{1}^{2}X''(x)T(t)+a_{2}^{2}X''(x)T(t-\tau)+d_{1}X(x)T(t)+d_{2}X(x)T(t-\tau),
\end{align*}
and hence
\begin{align*}
X(x)\Big[\ddot{T}(t)-d_{1}T(t)-d_{2}T(t-\tau)\Big]
=X''(x)\Big[a_{1}^{2}T(t)+a_{2}^{2}T(t-\tau)\Big].
\end{align*}
Formally separating the variables, we obtain
\begin{align*}
\frac{X''(x)}{X(x)}
=\frac{\ddot{T}(t)-d_{1}T(t)-d_{2}T(t-\tau)}{a_{1}^{2}T(t)+a_{2}^{2}T(t-\tau)}
=-\lambda^{2}.
\end{align*}
Therefore, we arrive at the following pair of second-order delayed ordinary differential equations:
\begin{align}\label{q.12}
\begin{cases}
X''(x)+\lambda^{2}X(x)=0,\\[1mm]
\ddot{T}(t)+\big(a_{1}^{2}\lambda^{2}-d_{1}\big)T(t)+\big(a_{2}^{2}\lambda^{2}-d_{2}\big)T(t-\tau)=0.
\end{cases}
\end{align}
Because $v_0$ satisfies homogeneous boundary conditions, we impose
\[
X(0)=0,\qquad X(L)=0.
\]
This gives a Sturm-Liouville eigenvalue problem which admits non-trivial solutions only for the eigenvalues
\[
\lambda_n=\frac{n\pi}{L},\qquad n=1,2,\dots,
\]
with corresponding eigenfunctions
\begin{align}\label{q.133}
X_n(x)=\sin\!\Big(\frac{n\pi}{L}x\Big),\qquad n=1,2,\dots.
\end{align}
Assume that
\[
a_{1}^{2}\lambda_n^{2}-d_{1}>0,\qquad a_{2}^{2}\lambda_n^{2}-d_{2}>0,\qquad n=1,2,\dots.
\]
We introduce the notation for $n \in \mathbb{N}$:
\begin{align}\label{def:omeganu}
\omega_{n}:=\sqrt{a_{1}^{2}\lambda_n^{2}-d_{1}}
=\sqrt{\Big(\frac{n\pi a_{1}}{L}\Big)^{2}-d_{1}},
\quad
\upsilon_{n}:=\sqrt{a_{2}^{2}\lambda_n^{2}-d_{2}}
=\sqrt{\Big(\frac{n\pi a_{2}}{L}\Big)^{2}-d_{2}}.
\end{align}
Then the second equation in \eqref{q.12} corresponding to $\lambda=\lambda_n$ takes the form
\begin{align}\label{q.13}
\ddot{T}_n(t)=-\omega_{n}^{2}T_n(t)-\upsilon_{n}^{2}T_n(t-\tau),\qquad n=1,2,\dots.
\end{align}

The initial conditions for each equation in \eqref{q.13} are obtained by expanding the history data $\Phi$ in the sine basis
\eqref{q.133}. For $t\in[-\tau,0]$ we write
\begin{equation}\label{q.14}
\Phi(t,x)
=\sum_{n=1}^{\infty}\Phi_n(t)\sin\!\Big(\frac{n\pi}{L}x\Big),
\end{equation}
where
\begin{equation*}
\Phi_n(t)=\frac{2}{L}\int_{0}^{L}\Phi(t,s)\sin\!\Big(\frac{n\pi}{L}s\Big)\,ds =\frac{2}{L}\int_{0}^{L}\big(\varphi(t,s)-G(t,s)\big)\sin\!\Big(\frac{n\pi}{L}s\Big)\,ds,
\end{equation*}
and, whenever $\partial_t\Phi$ exists for $t\in[-\tau,0]$,
\begin{equation}\label{q.14b}
\partial_t\Phi(t,x)
=\sum_{n=1}^{\infty}\dot{\Phi}_n(t)\sin\!\Big(\frac{n\pi}{L}x\Big),
\end{equation}
where
\begin{equation*}
\dot{\Phi}_n(t)=\frac{2}{L}\int_{0}^{L}\partial_t\Phi(t,s)\sin\!\Big(\frac{n\pi}{L}s\Big)\,ds=\frac{2}{L}\int_{0}^{L}\big(\partial_t\varphi(t,s)-\partial_t G(t,s)\big)\sin\!\Big(\frac{n\pi}{L}s\Big)\,ds.
\end{equation*}
Consequently, the modal functions $T_n$ in \eqref{q.13} satisfy the history conditions
\[
T_n(t)=\Phi_n(t),\qquad \dot{T}_n(t)=\dot{\Phi}_n(t),\qquad t\in[-\tau,0].
\]

To obtain an explicit representation for each $T_n$, we consider the following scalar second-order delay equation:
\begin{equation}\label{eq98}
\begin{cases}
\ddot{y}(t)=a\,y(t)+b\,y(t-\tau)+f(t),\quad t\in[0,T],\\
y(t)=\varphi(t),\quad \dot{y}(t)=\dot{\varphi}(t),\quad t\in[-\tau,0],
\end{cases}
\end{equation}
where $a,b\in\mathbb{R}$ and $f:[0,T]\to\mathbb{R}$ is given. In the following sections, we derive solution representations for both
the homogeneous and non-homogeneous cases.

\subsection{Representation of solution of \texorpdfstring{\eqref{eq98}}{(\ref{eq98})}} \label{subsec:eq98}

In this subsection, we present fundamental solutions for the homogeneous equation \eqref{eq98} (i.e.\ $f\equiv 0$) and derive
a Cauchy-type representation formula. Since solutions typically exhibit jump discontinuities in the second time derivative at the grid points
$t=k\tau$, we first introduce the corresponding solution concept.

\begin{definition}[Stepwise solution of \eqref{eq98}]\label{def:eq98-solution}
A function $y:[-\tau,T]\to\mathbb{R}$ is called a \emph{stepwise solution} of \eqref{eq98} if:
\begin{enumerate}[{\rm(i)}]
\item $y\in C([-\tau,T])$ and $\dot y\in C([-\tau,T])$;
\item for each $k=1,\dots,N$, the second derivative $\ddot y$ exists and is continuous on $I_k^\circ$ and admits a continuous extension to $\bar I_k$
(hence $\ddot y$ may jump at $t=k\tau$);
\item the identity $\ddot y(t)=a\,y(t)+b\,y(t-\tau)+f(t)$ holds pointwise for all $t\in I_k^\circ$, $k=1,\dots,N$;
\item $y(t)=\varphi(t)$ and $\dot y(t)=\dot\varphi(t)$ for all $t\in[-\tau,0]$.
\end{enumerate}
\end{definition}

We now turn to the homogeneous case $f\equiv 0$ in \eqref{eq98}. Following \cite{Abuasbeh}, we define the delayed perturbed hyperbolic cosine- and sine-type
functions as absolutely convergent series involving the truncation $(t)_+:=\max\{0,t\}$.

\begin{definition}[Delayed perturbation functions]\label{cossin}\cite{Abuasbeh}
Let $a,b\in\mathbb{R}$ and $\tau>0$. The \emph{delayed perturbation functions}
$C_{\tau}^{a,b},S_{\tau}^{a,b}:[0,+\infty)\to\mathbb{R}$ are defined by
\begin{align}\label{def:Ctau}
C_{\tau}^{a,b}(t)
&:=\sum_{k=0}^{\infty}\sum_{m=k}^{\infty}\binom{m}{k}a^{\,m-k}b^{\,k}\frac{(t-k\tau)_{+}^{2m}}{(2m)!},
\\
\label{def:Stau}
S_{\tau}^{a,b}(t)
&:=\sum_{k=0}^{\infty}\sum_{m=k}^{\infty}\binom{m}{k}a^{\,m-k}b^{\,k}\frac{(t-k\tau)_{+}^{2m+1}}{(2m+1)!}.
\end{align}
\end{definition}

In all convolution formulas below, $C_{\tau}^{a,b}$ and $S_{\tau}^{a,b}$ are extended by zero to negative arguments; thus $C_{\tau}^{a,b}(\eta)=S_{\tau}^{a,b}(\eta)=0$ for $\eta<0$, while $C_{\tau}^{a,b}(0)=1$ and $S_{\tau}^{a,b}(0)=0$.

\begin{theorem}[Fundamental solutions of the homogeneous delayed equation]\label{combined1}
Let $a,b\in\mathbb{R}$ and $\tau>0$. Then the functions $C_{\tau}^{a,b}$ and $S_{\tau}^{a,b}$ defined in
\eqref{def:Ctau}--\eqref{def:Stau} are stepwise classical solutions (in the sense of Definition~\ref{def:eq98-solution} with $f\equiv 0$)
of the homogeneous delay equation
\[
\ddot y(t)=a\,y(t)+b\,y(t-\tau),\qquad t\ge 0,
\]
and satisfy the initial conditions
\[
C_{\tau}^{a,b}(0)=1,\quad \dot C_{\tau}^{a,b}(0)=0,
\qquad
S_{\tau}^{a,b}(0)=0,\quad \dot S_{\tau}^{a,b}(0)=1.
\]
Moreover, for each integer $n\ge 0$, the identities
\[
\ddot C_{\tau}^{a,b}(t)=a\,C_{\tau}^{a,b}(t)+b\,C_{\tau}^{a,b}(t-\tau),
\qquad
\ddot S_{\tau}^{a,b}(t)=a\,S_{\tau}^{a,b}(t)+b\,S_{\tau}^{a,b}(t-\tau)
\]
hold pointwise for all $t\in(n\tau,(n+1)\tau)$.
\end{theorem}

\begin{proof}[Proof sketch (full proof in Appendix)]
Differentiate the defining series term-by-term on each interval $(n\tau,(n+1)\tau)$ (where each $(t-k\tau)_+$ is either identically zero
or a polynomial). After the index shift $m\mapsto m+1$, apply Pascal’s identity
$\binom{m+1}{k}=\binom{m}{k}+\binom{m}{k-1}$ to split the resulting sums into an $a$-part (reconstructing the original series at time $t$)
and a $b$-part (reconstructing the same series at time $t-\tau$). The jump behavior of $\ddot C_{\tau}^{a,b}$ and $\ddot S_{\tau}^{a,b}$
at $t=n\tau$ follows from the truncation $(\cdot)_+$. Details are given in Appendix~A.
\end{proof}

We next provide the Cauchy-type representation for the homogeneous problem \eqref{eq98} (with $f\equiv 0$).

\begin{theorem}[Cauchy formula for the homogeneous problem]\label{c1}
Let $\varphi\in C^{1}([-\tau,0],\mathbb{R})$. Then the stepwise classical solution of \eqref{eq98} with $f\equiv 0$
is given for $t\in[0,T]$ by
\begin{align}\label{f11}
y_{0}(t)
= C_{\tau}^{a,b}(t)\,\varphi(0)
+ S_{\tau}^{a,b}(t)\,\dot{\varphi}(0)
+ \int_{-\tau}^{0} S_{\tau}^{a,b}(t-\tau-s)\,b\,\varphi(s)\,ds.
\end{align}
\end{theorem}

\begin{proof}[Proof sketch (full proof in Appendix)]
On each interval $(n\tau,(n+1)\tau)$ one uses the method of steps: the term $y(t-\tau)$ is already known from the previous step,
so \eqref{eq98} reduces to a classical second-order ODE with a known inhomogeneity. The representation \eqref{f11} is obtained by matching
$y_0$ and $\dot y_0$ at $t=n\tau$ and using that $C_{\tau}^{a,b}$ and $S_{\tau}^{a,b}$ generate the corresponding homogeneous dynamics.
A complete induction argument is provided in Appendix~B.
\end{proof}

Next, we derive an explicit solution formula for the linear non-homogeneous delay equation \eqref{eq98} with a zero initial (history) condition.

\begin{theorem}\label{t2}
Let $f\in C([0,T],\mathbb{R})$. Consider \eqref{eq98} with the zero history
\[
y(t)=0,\qquad \dot y(t)=0,\qquad t\in[-\tau,0].
\]
Then the (stepwise classical) solution $y_p$ of \eqref{eq98} is given, for $t\ge 0$, by
\begin{align}\label{13}
y_p(t)=\int_{0}^{t} S_{\tau}^{a,b}(t-s)\,f(s)\,ds.
\end{align}
\end{theorem}

\begin{proof}[Proof sketch]
Define $y_p$ by \eqref{13} and set $y_p(t)=0$ for $t\in[-\tau,0]$. Since $S_{\tau}^{a,b}$ is $C^2$ on each open step interval
$(n\tau,(n+1)\tau)$ and $\dot S_{\tau}^{a,b}$ is continuous on $[0,\infty)$, differentiation under the integral sign is legitimate on each
open step interval. For $t$ away from the grid points $k\tau$ we compute
\[
\dot y_p(t)=\int_0^t \dot S_{\tau}^{a,b}(t-s)f(s)\,ds,\qquad
\ddot y_p(t)=\int_0^t \ddot S_{\tau}^{a,b}(t-s)f(s)\,ds+\dot S_{\tau}^{a,b}(0)f(t).
\]
Using $\dot S_{\tau}^{a,b}(0)=1$ and Theorem~\ref{combined1},
$\ddot S_{\tau}^{a,b}(\xi)=aS_{\tau}^{a,b}(\xi)+bS_{\tau}^{a,b}(\xi-\tau)$ for $\xi$ in the relevant open step intervals, we get
\[
\ddot y_p(t)=a\int_0^t S_{\tau}^{a,b}(t-s)f(s)\,ds
+b\int_0^t S_{\tau}^{a,b}(t-\tau-s)f(s)\,ds+f(t).
\]
Because $S_{\tau}^{a,b}(\eta)=0$ for $\eta<0$, the second integral equals $\int_0^{t-\tau} S_{\tau}^{a,b}(t-\tau-s)f(s)\,ds=y_p(t-\tau)$
(with the convention $y_p(t-\tau)=0$ for $t<\tau$). Hence $\ddot y_p(t)=a y_p(t)+b y_p(t-\tau)+f(t)$ holds pointwise for
$t\in(n\tau,(n+1)\tau)$, and $y_p(0)=\dot y_p(0)=0$ follows from \eqref{13}. The full argument is given in Appendix~\ref{app:proof-t2}.
\end{proof}

Combining Theorem~\ref{c1} and Theorem~\ref{t2}, we obtain the representation formula for the general non-homogeneous problem.

\begin{theorem}\label{j1}
Let $\varphi\in C^{1}([-\tau,0],\mathbb{R})$ and $f\in C([0,T],\mathbb{R})$. Then the stepwise classical solution $y$ of \eqref{eq98} is
\[
y(t)=
\begin{cases}
\varphi(t), & -\tau\le t\le 0,\\[1mm]
C_{\tau}^{a,b}(t)\,\varphi(0)+S_{\tau}^{a,b}(t)\,\dot\varphi(0)
+\displaystyle\int_{-\tau}^{0} S_{\tau}^{a,b}(t-\tau-s)\,b\,\varphi(s)\,ds
+\displaystyle\int_{0}^{t} S_{\tau}^{a,b}(t-s)\,f(s)\,ds, & t\ge 0.
\end{cases}
\]
\end{theorem}

\begin{proof}[Proof sketch]
Let $y_0$ be the homogeneous solution from Theorem~\ref{c1} (with history $\varphi$) and let $y_p$ be the zero-history particular solution from
Theorem~\ref{t2}. By linearity of \eqref{eq98}, the sum $y=y_0+y_p$ satisfies the forced equation on each open step interval and matches the
prescribed history on $[-\tau,0]$, because $y_p\equiv 0$ and $\dot y_p\equiv 0$ on $[-\tau,0]$. The full proof is immediate from
Theorems~\ref{c1} and \ref{t2}.
\end{proof}

\subsection{The solution of homogeneous equation}

In this subsection, we derive the solution representation for the linear homogeneous delayed wave equation \eqref{q.13}. For each $n=1,2,\dots$, consider
\begin{equation}\label{eq:Tn-problem}
\ddot T_n(t)=-\omega_n^{2}T_n(t)-\upsilon_n^{2}T_n(t-\tau),\quad t\in[0,T],
\end{equation}
together with the history conditions
\[
T_n(t)=\Phi_n(t),\quad \dot T_n(t)=\dot\Phi_n(t),\quad t\in[-\tau,0].
\]
Assume that $\Phi_n\in C^{1}([-\tau,0])$ (in particular, $\dot\Phi_n(0)$ is well-defined as the one-sided derivative at $0$). 
Then, by Theorem~\ref{c1} applied to \eqref{eq:Tn-problem} with $a=-\omega_n^{2}$ and $b=-\upsilon_n^{2}$, there exists a unique stepwise
solution $T_n$ (continuous with continuous first derivative, while $\ddot T_n$ may have jump discontinuities at $t=k\tau$), and it is given for $t\in [0,T]$ by
\begin{align}\label{eq:Tn-formula}
T_{n}(t)
&= C^{-\omega^{2}_{n},-\upsilon_{n}^{2}}_{\tau}(t)\,\Phi_{n}(0)
+ S^{-\omega^{2}_{n},-\upsilon_{n}^{2}}_{\tau}(t)\,\dot{\Phi}_{n}(0)\nonumber\\
&-\upsilon_{n}^{2}\int_{-\tau}^{0} S^{-\omega^{2}_{n},-\upsilon_{n}^{2}}_{\tau}(t-\tau-s)\,\Phi_{n}(s)\,ds .
\end{align}

Assuming the data are sufficiently smooth so that the Fourier series below converges in the sense required later, the solution $v_0$
of \eqref{q.11} satisfying the homogeneous boundary conditions and the history $v_0(t,\cdot)=\Phi(t,\cdot)$ for $t\in[-\tau,0]$ admits the representation
\begin{align}\label{eq:v0-series}
v_{0}(t,x)
&=\sum_{n=1}^{\infty} T_n(t)\,\sin\!\Big(\frac{n\pi}{L}x\Big),\qquad (t,x)\in[0,T]\times[0,L],
\end{align}
where the history coefficients are
\begin{align}\label{eq:Phi-n}
\Phi_n(t)
&:=\frac{2}{L}\int_{0}^{L}\Phi(t,s)\,\sin\!\Big(\frac{n\pi}{L}s\Big)\,ds
=\frac{2}{L}\int_{0}^{L}\big(\varphi(t,s)-G(t,s)\big)\,\sin\!\Big(\frac{n\pi}{L}s\Big)\,ds,
\quad t\in[-\tau,0].
\end{align}
Moreover, if $\partial_t\varphi$ and $\partial_t G$ exist and are continuous on $[-\tau,0]\times[0,L]$, then $\Phi_n\in C^{1}([-\tau,0])$ and
\[
\dot\Phi_n(t)=\frac{2}{L}\int_{0}^{L}\big(\partial_t\varphi(t,s)-\partial_t G(t,s)\big)\,\sin\!\Big(\frac{n\pi}{L}s\Big)\,ds,
\quad t\in[-\tau,0],
\]
so that the value $\dot\Phi_n(0)$ used in \eqref{eq:Tn-formula} is well-defined.

\section{Non-homogeneous equation}\label{sec5}

Next, we analyze the non-homogeneous equation \eqref{qq.12} with the right-hand side given by \eqref{uu1}, subject to homogeneous boundary
and homogeneous history conditions:
\begin{equation}\label{q.15}
\partial_{tt}v_1(t,x)
= a_{1}^{2}\partial_{xx}v_1(t,x)+a_{2}^{2}\partial_{xx}v_1(t-\tau,x)
+d_{1}v_1(t,x)+d_{2}v_1(t-\tau,x)+F(t,x),
\end{equation}
for $(t,x)\in[0,T]\times[0,L]$, with
\[
v_1(t,0)=0,\quad v_1(t,L)=0\quad \text{for }t\in [-\tau, T],
\quad
v_1(t,x)=0\quad \text{for }(t,x)\in[-\tau,0]\times[0,L].
\]
The source term $F$ is defined for $(t,x)\in[0,T]\times[0,L]$:
\begin{align}\label{F-def-sec5}
F(t,x)
&= f(t,x)
+ d_{1}\left( \mu_1(t) + \frac{x}{L}\big(\mu_2(t) - \mu_1(t)\big) \right) \nonumber\\
&+ d_{2}\left( \mu_1(t - \tau) + \frac{x}{L}\big(\mu_2(t - \tau) - \mu_1(t - \tau)\big) \right)- \ddot{\mu}_1(t) - \frac{x}{L}\big(\ddot{\mu}_2(t) - \ddot{\mu}_1(t)\big).
\end{align}
We expand $v_1$ into the sine eigenfunctions of the Dirichlet Sturm-Liouville problem:
\begin{equation}\label{q.16}
v_{1}(t,x)=\sum_{n=1}^{\infty}T_{1,n}(t)\sin\!\Big(\frac{n\pi}{L}x\Big).
\end{equation}
Substituting \eqref{q.16} into \eqref{q.15}, using orthogonality of $\sin(n\pi x/L)$, and comparing Fourier coefficients, we obtain for each
$n=1,2,\dots$ the forced delay differential equation
\begin{equation}\label{q.17}
\ddot T_{1,n}(t)= -\omega_n^{2}T_{1,n}(t)-\upsilon_n^{2}T_{1,n}(t-\tau)+F_n(t),\qquad t\in[0,T],
\end{equation}
with homogeneous history
\[
T_{1,n}(t)=0,\quad \dot T_{1,n}(t)=0,\quad t\in[-\tau,0],
\]
where the forcing coefficients are given by
\begin{equation}\label{Fn-def}
F_n(t):=\frac{2}{L}\int_{0}^{L}F(t,s)\,\sin\!\Big(\frac{n\pi}{L}s\Big)\,ds,\quad t\in[0,T].
\end{equation}
By Theorem~\ref{t2} applied to \eqref{q.17} with $a=-\omega_n^2$, $b=-\upsilon_n^2$ and $f=F_n$, the unique stepwise solution is
\begin{equation}\label{q.18}
T_{1,n}(t)=\int_{0}^{t} S_{\tau}^{-\omega_{n}^{2},-\upsilon_{n}^{2}}(t-s)\,F_{n}(s)\,ds,\qquad t\in [0,T].
\end{equation}
Therefore, the (formal) series representation of $v_1$ is
\begin{equation}\label{q.19}
v_{1}(t,x)
=\sum_{n=1}^{\infty}\left(\int_{0}^{t} S_{\tau}^{-\omega_{n}^{2},-\upsilon_{n}^{2}}(t-s)\,F_{n}(s)\,ds\right)
\sin\!\Big(\frac{n\pi}{L}x\Big),
\quad (t,x)\in[0,T]\times[0,L].
\end{equation}

\subsection{General case solution}\label{sec7}
Recalling the decomposition $v=v_0+v_1+G$, we combine the homogeneous part (Section~\ref{sec4}) with the forced part
(Section~\ref{sec5}). Formally, for $(t,x)\in[0,T]\times[0,L]$,
\begin{align}\label{q.20}
v(t,x)&
=\sum_{n=1}^{\infty}\Bigg(
C^{-\omega^{2}_{n},-\upsilon_{n}^{2}}_{\tau}(t)\,\Phi_{n}(0)
+S^{-\omega^{2}_{n},-\upsilon_{n}^{2}}_{\tau}(t)\,\dot{\Phi}_{n}(0)\nonumber\\
&-\upsilon_{n}^{2}\int_{-\tau}^{0}S^{-\omega^{2}_{n},-\upsilon_{n}^{2}}_{\tau}(t-\tau-s)\,\Phi_{n}(s)\,ds \nonumber\\
&+\int_{0}^{t} S_{\tau}^{-\omega_{n}^{2},-\upsilon_{n}^{2}}(t-s)\,F_{n}(s)\,ds
\Bigg)\sin\!\Big(\frac{n\pi}{L}x\Big)
+G(t,x),
\end{align}
where $\Phi_n$ is the sine coefficient of the history $\Phi=\varphi-G$ on $[-\tau,0]$ (cf.\ \eqref{eq:Phi-n}) and $F_n$ is defined in \eqref{Fn-def}. The solution of the original problem is then recovered by $u(t,x)=e^{-\alpha x}v(t,x)$, with $\alpha=b_1/(2a_1^2)$ in the non-degenerate compatible case.

\section{Convergence of the Fourier series}\label{sec8}

\begin{theorem}\label{thm:conv}
Let $T>0$, $\tau>0$ and set $K:=\left\lceil \frac{T}{\tau}\right\rceil$.
Assume that the modal frequencies satisfy
\[
\omega_n^2=\Big(\frac{n\pi a_1}{L}\Big)^2-d_1>0,\qquad 
\upsilon_n^2=\Big(\frac{n\pi a_2}{L}\Big)^2-d_2>0
\]
for all sufficiently large $n$.
Let $\Phi_n$ be the Fourier coefficients of the history $\Phi=\varphi-G$ on $[-\tau,0]$ (see \eqref{eq:Phi-n})
and let $F_n$ be the Fourier coefficients of the source term $F$ on $[0,T]$ (see \eqref{Fn-def}).
Assume that for some fixed $\alpha>0$,
\[
\lim_{n\to\infty}\Big(|\Phi_n(0)|+|\dot\Phi_n(0)|\Big)\,n^{K+3+\alpha}=0,
\]
\[
\lim_{n\to\infty}\max_{-\tau\le s\le 0}|\Phi_n(s)|\,n^{K+4+\alpha}=0,
\]
\[
\lim_{n\to\infty}\max_{0\le s\le T}|F_n(s)|\,n^{K+2+\alpha}=0.
\]

Equivalently, the three decay rates are
\[
|\Phi_n(0)|+|\dot\Phi_n(0)|=o(n^{-K-3-\alpha}),\qquad
\max_{[-\tau,0]}|\Phi_n|=o(n^{-K-4-\alpha}),\qquad
\max_{[0,T]}|F_n|=o(n^{-K-2-\alpha}).
\]

Then the Fourier series representation \eqref{q.20} for $v$ converges absolutely and uniformly on $[0,T]\times[0,L]$.
Moreover,
\[
v,\; v_x,\; v_{xx},\; v_t,\; v_{tx}\in C([0,T]\times[0,L]),
\]
and $v_{tt}$ exists and is continuous on each open step strip $((m-1)\tau,m\tau)\times[0,L]$ ($m=1,\dots,K$), with finite one-sided limits at
$t=m\tau$. All these derivatives can be obtained by termwise differentiation of the Fourier series on the corresponding domains, and the
differentiated series converge absolutely and uniformly on those domains.
\end{theorem}

\begin{proof}
Write the series \eqref{q.20} (without $G$) as
\[
v(t,x)-G(t,x)=S_1(t,x)+S_2(t,x)+S_3(t,x),
\]
where
\begin{align*}
S_1(t,x)&=\sum_{n=1}^\infty A_n(t)\sin\!\Big(\frac{n\pi}{L}x\Big),\qquad
A_n(t)=C_\tau^{-\omega_n^2,-\upsilon_n^2}(t)\Phi_n(0)+S_\tau^{-\omega_n^2,-\upsilon_n^2}(t)\dot\Phi_n(0),\\
S_2(t,x)&=\sum_{n=1}^\infty B_n(t)\sin\!\Big(\frac{n\pi}{L}x\Big),\qquad
B_n(t)=-\upsilon_n^2\int_{-\tau}^0 S_\tau^{-\omega_n^2,-\upsilon_n^2}(t-\tau-s)\,\Phi_n(s)\,ds,\\
S_3(t,x)&=\sum_{n=1}^\infty C_n(t)\sin\!\Big(\frac{n\pi}{L}x\Big),\qquad
C_n(t)=\int_0^t S_\tau^{-\omega_n^2,-\upsilon_n^2}(t-s)\,F_n(s)\,ds .
\end{align*}

\medskip
\noindent\textbf{Step 1: Kernel bounds on $[0,T]$.}
Fix $r\in\{0,1,2\}$. On each open step interval $((m-1)\tau,m\tau)$ the functions
$C_\tau^{-\omega_n^2,-\upsilon_n^2}$ and $S_\tau^{-\omega_n^2,-\upsilon_n^2}$ are $C^2$, and only finitely many delay-interaction terms
are active (at most $m$ terms). Equivalently, by the method of steps these functions solve, on each such interval, a non-delayed second order
ODE with forcing determined by previous intervals; hence they can be written as linear combinations of $\sin(\omega_n t)$ and $\cos(\omega_n t)$
with coefficients that are polynomials in $t$ whose degrees are bounded by $m$.
Since $m\le K$ on $[0,T]$ and $\omega_n,\upsilon_n\sim n$, there exist constants $M_r=M_r(T)$ such that, for all $n$ and all $t\in[0,T]$,
\begin{align}\label{eq:kernel-bounds}
\big|\partial_t^r C_\tau^{-\omega_n^2,-\upsilon_n^2}(t)\big| \le M_r\,n^{K+r},\qquad
\big|\partial_t^r S_\tau^{-\omega_n^2,-\upsilon_n^2}(t)\big| \le M_r\,n^{K+r-1}.
\end{align}
In addition, using $\upsilon_n^2\sim n^2$ and \eqref{eq:kernel-bounds},
\begin{equation}\label{eq:kernel-bounds-bS}
\big|\upsilon_n^2\,\partial_t^r S_\tau^{-\omega_n^2,-\upsilon_n^2}(t)\big|\le M_r\,n^{K+r+1},\qquad t\in[0,T].
\end{equation}
(Only the polynomial dependence on $n$ is needed below; the constants depend on $T$ but not on $n$.)

\medskip
\noindent\textbf{Step 2: Absolute and uniform convergence of $S_1,S_2,S_3$.}
From \eqref{eq:kernel-bounds} with $r=0$ we get
\[
|A_n(t)|\le M_0\Big(n^K|\Phi_n(0)|+n^{K-1}|\dot\Phi_n(0)|\Big)\le M_0\,n^K\Big(|\Phi_n(0)|+|\dot\Phi_n(0)|\Big).
\]
By the first assumption, for $n$ large we have $|\Phi_n(0)|+|\dot\Phi_n(0)|\le c\,n^{-K-3-\alpha}$, hence
$|A_n(t)|\le c\,n^{-3-\alpha}$, uniformly in $t\in[0,T]$. Since $\sum_{n\ge1}n^{-3-\alpha}<\infty$, the Weierstrass M--test yields
uniform and absolute convergence of $S_1$ on $[0,T]\times[0,L]$.

For $B_n$, by \eqref{eq:kernel-bounds-bS} with $r=0$,
\[
|B_n(t)|
\le \int_{-\tau}^0 \big|\upsilon_n^2 S_\tau^{-\omega_n^2,-\upsilon_n^2}(t-\tau-s)\big|\,|\Phi_n(s)|\,ds
\le \tau\,M_0\,n^{K+1}\max_{-\tau\le s\le 0}|\Phi_n(s)|.
\]
By the second assumption, $\max_{[-\tau,0]}|\Phi_n(s)|\le c\,n^{-K-4-\alpha}$ for $n$ large, hence
$|B_n(t)|\le c\,n^{-3-\alpha}$ uniformly in $t\in[0,T]$. Since $\sum_{n\ge1}n^{-3-\alpha}<\infty$, $S_2$ converges absolutely and uniformly.

For $C_n$, using \eqref{eq:kernel-bounds} with $r=0$,
\[
|C_n(t)|
\le \int_0^t \big|S_\tau^{-\omega_n^2,-\upsilon_n^2}(t-s)\big|\,|F_n(s)|\,ds
\le T\,M_0\,n^{K-1}\max_{0\le s\le T}|F_n(s)|.
\]
By the third assumption, $\max_{[0,T]}|F_n(s)|\le c\,n^{-K-2-\alpha}$ for $n$ large, hence
$|C_n(t)|\le c\,n^{-3-\alpha}$ uniformly in $t\in[0,T]$. Thus $S_3$ converges absolutely and uniformly.

Combining the three parts, $v$ given by \eqref{q.20} converges absolutely and uniformly on $[0,T]\times[0,L]$.

\medskip
\noindent\textbf{Step 3: Termwise differentiation and convergence of derivative series.}
Differentiation in $x$ multiplies the $n$th sine mode by factors of order $n$ (for $v_x$) and $n^2$ (for $v_{xx}$).
Using the bounds above, we obtain dominating summable majorants for the differentiated coefficients.
For example, for $S_1$,
\[
\left|\partial_{xx}\big(A_n(t)\sin(\tfrac{n\pi}{L}x)\big)\right|
\le \Big(\tfrac{n\pi}{L}\Big)^2 |A_n(t)|
\le c\,n^2\cdot n^{-3-\alpha}=c\,n^{-1-\alpha},
\]
which is summable for $\alpha>0$, hence $\partial_{xx}S_1$ converges uniformly; similarly for $\partial_x S_1$.
The same reasoning applies to $S_2$ and $S_3$ with their corresponding bounds.

For first time derivatives, use \eqref{eq:kernel-bounds}--\eqref{eq:kernel-bounds-bS} with $r=1$ to obtain bounds for $\dot A_n,\dot B_n,\dot C_n$,
and proceed exactly as above (again the M--test applies).

For the second time derivative, \eqref{eq:kernel-bounds}--\eqref{eq:kernel-bounds-bS} with $r=2$ gives the majorant $c n^{-1-\alpha}$ on each open step strip. In addition, differentiating the forced convolution twice produces the endpoint term $F_n(t)$, and the zero-extension convention in the history convolution can produce on the first step the term $\upsilon_n^2\Phi_n(t-\tau)$; these are bounded by $c n^{-K-2-\alpha}$ and are therefore summable.

At the grid points $t=m\tau$ the termwise limits may differ from the left and right because the kernels themselves
have stepwise second derivatives, so $v_{tt}$ is only piecewise continuous in $t$, while $v$ and $v_t$ remain continuous.

Therefore, $v, v_x, v_{xx}, v_t, v_{tx}$ are continuous on $[0,T]\times[0,L]$, and $v_{tt}$ is continuous on each open step strip with one-sided limits
at step endpoints. All these derivatives are obtained by termwise differentiation on the corresponding domains.
\end{proof}

\begin{remark}\label{rem:academic_contribution}
We compare \eqref{q.1} with the closest explicit-representation results. Khusainov, Pokojovy and Azizbayov~\cite{Khusainov} studied the one-dimensional pure-delay problem
\[
\eta_{tt}(t,x)=a^2\eta_{xx}(t-\tau,x)+b\eta_x(t-\tau,x)+d\eta(t-\tau,x)+g(t,x),
\]
with non-homogeneous Dirichlet data. At the level of the original equation \eqref{q.1}, this corresponds to suppressing the present operator, that is, taking $a_1=b_1=c_1=0$ and identifying the remaining delayed coefficients. This is a degenerate case for the exponential reduction used in \eqref{q.3}--\eqref{q.7}, so it should be understood as a related limiting case of the model, not as a direct consequence of the non-degenerate reduction theorem.

Rodríguez et al.~\cite{Rodríguez} considered the one-dimensional mixed problem
\[
u_{tt}(t,x)=a^2u_{xx}(t,x)-b\,u(t-\tau,x)
\]
with homogeneous Dirichlet boundary conditions. The case $d=1$ in Jornet~\cite{Jornet} treats
\[
u_{tt}(t,x)=a_1^2u_{xx}(t,x)+a_2^2u_{xx}(t-\tau,x)+b_1u(t,x)+b_2u(t-\tau,x),
\]
also with homogeneous Dirichlet data. These two equations fit into the reduced equation \eqref{q.6} after choosing homogeneous boundary data, zero forcing, and suitable values of $d_1,d_2,a_1,a_2$. In this sense, the closest overlap with \cite{Rodríguez,Jornet} is through the reduced problem, where no first-order spatial derivatives are present.

The additional point in the present paper is the treatment of the original equation \eqref{q.1}, which contains both present and delayed first-order spatial terms. A single exponential change of variables removes both of them only under the compatibility condition \eqref{q.5}. After this reduction, the modal equations contain both the present and delayed spectral parts and lead to the explicit formula \eqref{q.20}, including the history term and the forcing term. The convergence assumptions are stated for stepwise classical solutions, so possible jumps of $u_{tt}$ at the delay nodes $t=k\tau$ are allowed.

\end{remark}

\section{Numerical Example}\label{sec9}

In this section we illustrate the analytical representation by a numerical example.
We consider \eqref{e.1} on $[0,T]\times[0,\pi]$ with $\tau=2$, $T=10$ and $L=\pi$, and choose
\[
a_{1}=1,\quad a_{2}=2,\quad b_{1}=1,\quad b_{2}=4,\quad c_{1}=-\frac{3}{4},\quad c_{2}=27,
\quad g(t,x)=e^{-x/2}\,t.
\]
The boundary and history data are prescribed by
\[
u(t,0)=1,\quad u(t,\pi)=0,\quad t\in[-2,10],
\]
\[
u(t,x)=e^{-x/2},\quad (t,x)\in[-2,0]\times[0,\pi].
\]
Hence the delayed wave problem reads, for $(t,x)\in[0,10]\times(0,\pi)$,
\begin{equation}\label{w1}
\begin{aligned}
\partial_{tt}u(t,x)&=\partial_{xx}u(t,x)+4\,\partial_{xx}u(t-2,x)
+\partial_x u(t,x)+4\,\partial_x u(t-2,x)\\
&-\frac{3}{4}u(t,x)+27u(t-2,x)+e^{-x/2}t,
\end{aligned}
\end{equation}
together with the above boundary conditions on $[-2,10]$ and the history on $[-2,0]$.

\medskip
We apply the transformation $u(t,x)=e^{-x/2}v(t,x)$, which yields the reduced delayed wave equation:
\begin{equation}\label{w2}
\begin{cases}
\partial_{tt}v(t,x)=\partial_{xx}v(t,x)+4\,\partial_{xx}v(t-2,x)-v(t,x)+26\,v(t-2,x)+t,
& (t,x)\in[0,10]\times[0,\pi],\\[1mm]
v(t,0)=1,\quad v(t,\pi)=0,\quad t\in[-2,10],\\
v(t,x)=1,\quad (t,x)\in[-2,0]\times[0,\pi].
\end{cases}
\end{equation}
To homogenize the boundary conditions we introduce
\[
G(x):=1-\frac{x}{\pi},\qquad v(t,x)=w(t,x)+G(x),
\]
so that $w(t,0)=w(t,\pi)=0$ for $t\in[-2,10]$ and the history becomes
\[
w(t,x)=1-G(x)=\frac{x}{\pi},\quad (t,x)\in[-2,0]\times[0,\pi].
\]
Expanding in the sine basis,
\[
w(t,x)=\sum_{n=1}^\infty q_n(t)\sin(nx),
\]
the history coefficients are
\[
\Phi_n:=\frac{2}{\pi}\int_0^\pi \frac{s}{\pi}\sin(ns)\,ds
=\frac{2}{\pi}\frac{(-1)^{n+1}}{n},\quad n\in\mathbb{N}.
\]
Moreover, the Fourier coefficients of the lifted source term
\[
F(t,x):=t+25\Big(1-\frac{x}{\pi}\Big)
\]
are given by
\[
F_n(t):=\frac{2}{\pi}\int_0^\pi F(t,s)\sin(ns)\,ds
=\frac{2t\big(1-(-1)^n\big)+50}{\pi n}.
\]
For each mode $n\ge 1$ we obtain the scalar delay equation (with $\tau=2$)
\[
\ddot q_n(t)=a_n q_n(t)+b_n q_n(t-2)+F_n(t),\qquad
a_n:=-(n^2+1),\quad b_n:=26-4n^2,
\]
and hence, by Theorems~\ref{c1} and~\ref{t2}, the solution can be represented in terms of the delayed perturbation functions
$C_{2}^{a_n,b_n}$ and $S_{2}^{a_n,b_n}$.

In Figure \ref{fig:uN_surface_heatmap}, we use the truncated approximation
\[
v_N(t,x):=G(x)+\sum_{n=1}^N q_n(t)\sin(nx),\qquad
u_N(t,x):=e^{-x/2}v_N(t,x),
\]
where both the Fourier series and the defining series of $C_{2}^{a_n,b_n}$ and $S_{2}^{a_n,b_n}$
are truncated. We evaluate $u_N$ on a uniform grid and plot its surface on $[0,10]\times[0,\pi]$.
As expected for stepwise classical solutions of delay equations, $u_N$ and $\partial_t u_N$ are continuous in $t$,
while $\partial_{tt}u_N$ may exhibit jump discontinuities at $t=2,4,6,8,10$.

\begin{figure}[h!]
\centering
\includegraphics[width=0.6\textwidth]{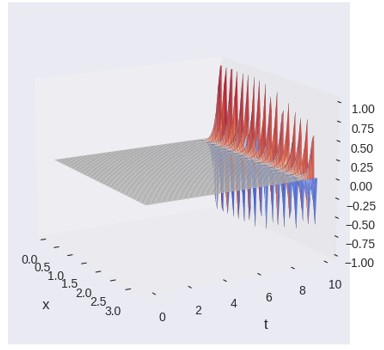}\hfill
\includegraphics[width=0.6\textwidth]{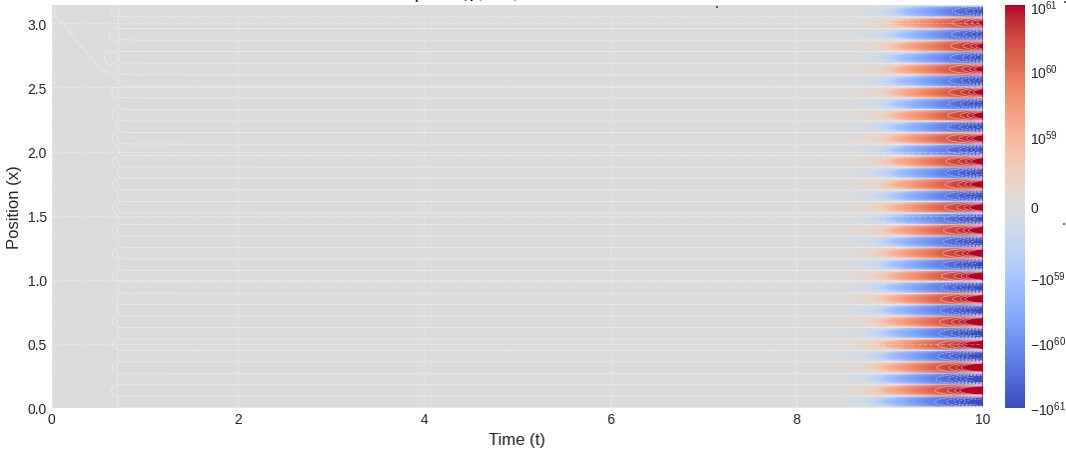}
\caption{Truncated approximation $u_N(t,x)$ for the numerical example.
\textbf{Top:} 3D surface plot of $u_N(t,x)$ over $(t,x)\in[0,T]\times[0,\pi]$.
\textbf{Bottom:} 2D heatmap (top view) of the same data, where colors encode the solution magnitude.
For visualization we use percentile clipping and a symmetric-log color normalization to reveal the structure in the presence of large-amplitude peaks.}
\label{fig:uN_surface_heatmap}
\end{figure}


\section{Conclusion}\label{sec:conclusion}

The main findings of this study can be summarized as follows:

\begin{itemize}
\item An explicit analytical representation was derived for a generalized one-dimensional linear wave equation with a single time delay, covering non-homogeneous Dirichlet boundary data and prescribed history functions.
\item By separation of variables and Sturm--Liouville spectral decomposition, the delayed PDE was reduced to a countable family of second-order linear delay differential equations for the Fourier modes, whose fundamental solutions are expressed through the delayed perturbation functions $C_{\tau}^{a,b}(t)$ and $S_{\tau}^{a,b}(t)$.
\item The approach extends earlier formulations (e.g.\ Khusainov et al.~\cite{Khusainov}) by providing an explicit Fourier-based solution formula for a broader class of delayed wave models with both delayed and non-delayed spatial operators.
\item Sufficient conditions were established to guarantee absolute and uniform convergence of the Fourier series representation and of the termwise differentiated series in $x$; in time, the solution is stepwise classical, i.e.\ $u$ and $\partial_t u$ are continuous while $\partial_{tt}u$ may have jump discontinuities at the delay grid points.
\item A numerical example illustrated how the truncated analytical series can be evaluated effectively and visualized, supporting the applicability of the representation for simulation and qualitative analysis of delayed wave dynamics.
\end{itemize}

\section*{Declarations}
\textbf{Conflicts of Interest} The authors declare no conflict of interest.

\section*{Data Availability}
 No datasets were generated or analysed during the current study.

\section*{Author Contributions}
Javad A. Asadzade: Conceptualization, Formal analysis, Writing -Original draft preparation; Jasarat J. Gasimov: Conceptualization, Formal analysis, Writing - Original draft preparation; Nazim I. Mahmudov: Conceptualization, Methodology, Supervision; Ismail T. Huseynov: Methodology, Investigation, Software,  Writing - Original draft preparation, Writing - Reviewing and Editing, Supervision.







\appendix

\section{Proof of Theorem~\ref{combined1}}\label{app:proof-combined1}

In this appendix we give a complete proof that the delayed perturbation functions $C_{\tau}^{a,b}$ and $S_{\tau}^{a,b}$ from
Definition~\ref{cossin} solve the linear homogeneous second-order delayed ordinary differential equation:
\begin{equation}\label{app:eq-hom}
\ddot y(t)=a\,y(t)+b\,y(t-\tau),\qquad t\in [0,T],
\end{equation}
in the stepwise sense (i.e.\ pointwise on every open step interval).

Fix an integer $n\ge 0$ and let $t\in(n\tau,(n+1)\tau)$. Then $(t-k\tau)_+=t-k\tau$ holds for $k=0,1,\dots,n$, and
$(t-k\tau)_+=0$ holds for all $k\ge n+1$. Hence, on the interval $(n\tau,(n+1)\tau)$ the defining series reduce to the finite $k$--sums
\begin{align}
C_{\tau}^{a,b}(t)
&=\sum_{k=0}^{n}\sum_{m=k}^{\infty}\binom{m}{k}a^{\,m-k}b^{\,k}\frac{(t-k\tau)^{2m}}{(2m)!}, \label{app:C-finite}\\
S_{\tau}^{a,b}(t)
&=\sum_{k=0}^{n}\sum_{m=k}^{\infty}\binom{m}{k}a^{\,m-k}b^{\,k}\frac{(t-k\tau)^{2m+1}}{(2m+1)!}. \label{app:S-finite}
\end{align}
For fixed $k$, the inner series in \eqref{app:C-finite} and \eqref{app:S-finite} are entire power series in $(t-k\tau)$, hence can be
differentiated termwise on $(n\tau,(n+1)\tau)$. Since the outer sum is finite, it follows that $C_{\tau}^{a,b}$ and $S_{\tau}^{a,b}$
are $C^{\infty}$ on each open step interval $(n\tau,(n+1)\tau)$, and therefore $C^{2}$ there.
Moreover, because new $k$--terms become active exactly at the grid points $t=n\tau$, the second derivatives may have jump discontinuities at
$t=n\tau$ (while $C_{\tau}^{a,b}$ and $S_{\tau}^{a,b}$ and their first derivatives remain continuous). This is precisely the usual
stepwise regularity for delay equations.

We first verify \eqref{app:eq-hom} for $C_{\tau}^{a,b}$. Fix an integer $n\ge 0$ and $t\in(n\tau,(n+1)\tau)$. Differentiating
\eqref{app:C-finite} twice gives
\begin{align*}
\ddot C_{\tau}^{a,b}(t)
&=\sum_{k=0}^{n}\sum_{m=k}^{\infty}\binom{m}{k}a^{\,m-k}b^{\,k}\frac{(t-k\tau)^{2m-2}}{(2m-2)!}.
\end{align*}
Perform the index shift $r=m-1$ (so $r\ge k-1$) and use $(2m-2)!=(2r)!$ to obtain
\begin{align*}
\ddot C_{\tau}^{a,b}(t)
&=\sum_{k=0}^{n}\sum_{r=k-1}^{\infty}\binom{r+1}{k}a^{\,r+1-k}b^{\,k}\frac{(t-k\tau)^{2r}}{(2r)!}.
\end{align*}
Apply Pascal’s identity $\binom{r+1}{k}=\binom{r}{k}+\binom{r}{k-1}$ and split the sum:
\begin{align*}
\ddot C_{\tau}^{a,b}(t)
&=\sum_{k=0}^{n}\sum_{r=k}^{\infty}\binom{r}{k}a^{\,r+1-k}b^{\,k}\frac{(t-k\tau)^{2r}}{(2r)!}
+\sum_{k=1}^{n}\sum_{r=k-1}^{\infty}\binom{r}{k-1}a^{\,r+1-k}b^{\,k}\frac{(t-k\tau)^{2r}}{(2r)!}\\
&=: I_1(t)+I_2(t).
\end{align*}
For $I_1$, factor out $a$ and recognize \eqref{app:C-finite}:
\begin{align*}
I_1(t)
&=a\sum_{k=0}^{n}\sum_{r=k}^{\infty}\binom{r}{k}a^{\,r-k}b^{\,k}\frac{(t-k\tau)^{2r}}{(2r)!}
=a\,C_{\tau}^{a,b}(t).
\end{align*}
For $I_2$, set $j=k-1$ (so $j=0,1,\dots,n-1$) to obtain
\begin{align*}
I_2(t)
&=\sum_{j=0}^{n-1}\sum_{r=j}^{\infty}\binom{r}{j}a^{\,r-j}b^{\,j+1}\frac{(t-(j+1)\tau)^{2r}}{(2r)!}
=b\sum_{j=0}^{n-1}\sum_{r=j}^{\infty}\binom{r}{j}a^{\,r-j}b^{\,j}\frac{(t-\tau-j\tau)^{2r}}{(2r)!}.
\end{align*}
Since $t\in(n\tau,(n+1)\tau)$ implies $t-\tau\in((n-1)\tau,n\tau)$, the truncation gives the representation
\[
C_{\tau}^{a,b}(t-\tau)
=\sum_{j=0}^{n-1}\sum_{r=j}^{\infty}\binom{r}{j}a^{\,r-j}b^{\,j}\frac{(t-\tau-j\tau)^{2r}}{(2r)!},
\]
and thus $I_2(t)=b\,C_{\tau}^{a,b}(t-\tau)$. Consequently,
\[
\ddot C_{\tau}^{a,b}(t)=a\,C_{\tau}^{a,b}(t)+b\,C_{\tau}^{a,b}(t-\tau),
\qquad t\in(n\tau,(n+1)\tau),
\]
for every integer $n\ge 0$, proving \eqref{app:eq-hom} for $C_{\tau}^{a,b}$ pointwise on each open step interval.

The proof for $S_{\tau}^{a,b}$ is analogous. Fix an integer $n\ge 0$ and $t\in(n\tau,(n+1)\tau)$. Differentiating \eqref{app:S-finite}
twice yields
\begin{align*}
\ddot S_{\tau}^{a,b}(t)
&=\sum_{k=0}^{n}\sum_{m=k}^{\infty}\binom{m}{k}a^{\,m-k}b^{\,k}\frac{(t-k\tau)^{2m-1}}{(2m-1)!}.
\end{align*}
With the shift $r=m-1$ one obtains
\begin{align*}
\ddot S_{\tau}^{a,b}(t)
&=\sum_{k=0}^{n}\sum_{r=k-1}^{\infty}\binom{r+1}{k}a^{\,r+1-k}b^{\,k}\frac{(t-k\tau)^{2r+1}}{(2r+1)!}.
\end{align*}
Using Pascal’s identity and repeating the splitting and reindexing as above yields
\[
\ddot S_{\tau}^{a,b}(t)=a\,S_{\tau}^{a,b}(t)+b\,S_{\tau}^{a,b}(t-\tau),
\qquad t\in(n\tau,(n+1)\tau),
\]
for every integer $n\ge 0$.

Finally, the initial values follow directly from Definition~\ref{cossin}. At $t=0$ the only nonzero term in $C_{\tau}^{a,b}(0)$ is $k=0$, $m=0$,
hence $C_{\tau}^{a,b}(0)=1$, and $\dot C_{\tau}^{a,b}(0)=0$ since all terms in $C_{\tau}^{a,b}$ are even powers. Likewise,
$S_{\tau}^{a,b}(0)=0$ and $\dot S_{\tau}^{a,b}(0)=1$ follow from the leading term $m=0$ in \eqref{def:Stau}, i.e.\ $S_{\tau}^{a,b}(t)=t+O(t^{3})$.
This completes the proof of Theorem~\ref{combined1}.
\qed

\section{Proof of Theorem~\ref{c1}}\label{app:proof-c1}

Let $\varphi\in C^{1}([-\tau,0],\mathbb{R})$ be given and define $y_0(t)$ for $t\in[0,T]$ by \eqref{f11}, i.e.
\[
y_{0}(t)= C_{\tau}^{a,b}(t)\,\varphi(0) + S_{\tau}^{a,b}(t)\,\dot{\varphi}(0)
+ \int_{-\tau}^{0} S_{\tau}^{a,b}(t-\tau-s)\,b\,\varphi(s)\,ds.
\]
We show that $y_0$ is the stepwise classical solution of \eqref{eq98} with $f\equiv 0$.

At $t=0$, using $C_{\tau}^{a,b}(0)=1$, $S_{\tau}^{a,b}(0)=0$ and $S_{\tau}^{a,b}(-\tau-s)=0$ for $s\in[-\tau,0]$ (because the argument is $\le 0$
and the truncation $(\cdot)_+$ annihilates every term of $S_{\tau}^{a,b}$), we obtain
\[
y_0(0)=\varphi(0).
\]
Differentiating \eqref{f11} with respect to $t$ gives
\[
\dot y_0(t)= \dot C_{\tau}^{a,b}(t)\,\varphi(0) + \dot S_{\tau}^{a,b}(t)\,\dot{\varphi}(0)
+ \int_{-\tau}^{0} \dot S_{\tau}^{a,b}(t-\tau-s)\,b\,\varphi(s)\,ds,
\]
and at $t=0$, since $\dot C_{\tau}^{a,b}(0)=0$, $\dot S_{\tau}^{a,b}(0)=1$, and $\dot S_{\tau}^{a,b}(-\tau-s)=0$ for $s\in[-\tau,0]$, we get
\[
\dot y_0(0)=\dot\varphi(0).
\]
As usual for delay equations, the solution is prescribed by the history on $[-\tau,0]$ and by \eqref{f11} on $[0,T]$.

By Theorem~\ref{combined1}, $C_{\tau}^{a,b}$ and $S_{\tau}^{a,b}$ are $C^{2}$ on each open step interval $(n\tau,(n+1)\tau)$ and have
continuous first derivatives on $[0,\infty)$. Since $\varphi$ is continuous on $[-\tau,0]$, the integral term in \eqref{f11} is continuously
differentiable on $[0,T]$ and twice differentiable on each open step interval. Hence $y_0$ is a stepwise classical function on $[0,T]$.

Fix an integer $n\ge 0$ and let $t\in(n\tau,(n+1)\tau)$. Differentiating \eqref{f11} twice and using differentiation under the integral sign
(valid on each open step interval since the integrand is $C^{2}$ in $t$ there) yields
\begin{align*}
\ddot y_0(t)
&=\ddot C_{\tau}^{a,b}(t)\,\varphi(0)+\ddot S_{\tau}^{a,b}(t)\,\dot\varphi(0)
+ \int_{-\tau}^{0}\ddot S_{\tau}^{a,b}(t-\tau-s)\,b\,\varphi(s)\,ds.
\end{align*}
Now apply Theorem~\ref{combined1} to $C_{\tau}^{a,b}$ and $S_{\tau}^{a,b}$ at time $t$ and also at time $t-\tau-s$ (the latter is legitimate
because for any fixed $t$ the integrand vanishes whenever $t-\tau-s\le 0$):
\[
\ddot C_{\tau}^{a,b}(t)=a\,C_{\tau}^{a,b}(t)+b\,C_{\tau}^{a,b}(t-\tau),\qquad
\ddot S_{\tau}^{a,b}(t)=a\,S_{\tau}^{a,b}(t)+b\,S_{\tau}^{a,b}(t-\tau),
\]
and
\[
\ddot S_{\tau}^{a,b}(t-\tau-s)=a\,S_{\tau}^{a,b}(t-\tau-s)+b\,S_{\tau}^{a,b}(t-2\tau-s).
\]
Substituting these relations into $\ddot y_0(t)$ gives
\begin{align*}
\ddot y_0(t)
&=a\Big(C_{\tau}^{a,b}(t)\varphi(0)+S_{\tau}^{a,b}(t)\dot\varphi(0)+\int_{-\tau}^{0}S_{\tau}^{a,b}(t-\tau-s)\,b\,\varphi(s)\,ds\Big)\\
&\quad +b\Big(C_{\tau}^{a,b}(t-\tau)\varphi(0)+S_{\tau}^{a,b}(t-\tau)\dot\varphi(0)+\int_{-\tau}^{0}S_{\tau}^{a,b}(t-2\tau-s)\,b\,\varphi(s)\,ds\Big).
\end{align*}
The expression in the first parentheses is exactly $y_0(t)$. The expression in the second parentheses is exactly $y_0(t-\tau)$, because
\eqref{f11} with $t$ replaced by $t-\tau$ reads
\[
y_0(t-\tau)=C_{\tau}^{a,b}(t-\tau)\varphi(0)+S_{\tau}^{a,b}(t-\tau)\dot\varphi(0)+\int_{-\tau}^{0}S_{\tau}^{a,b}(t-2\tau-s)\,b\,\varphi(s)\,ds.
\]
Therefore,
\[
\ddot y_0(t)=a\,y_0(t)+b\,y_0(t-\tau),\qquad t\in(n\tau,(n+1)\tau),
\]
for every integer $n\ge 0$, i.e.\ $y_0$ satisfies \eqref{eq98} with $f\equiv 0$ pointwise on each open step interval.

\section{Proof of Theorem~\ref{t2}}\label{app:proof-t2}

Let $f\in C([0,T],\mathbb{R})$. Define $y_p(t)=0$ for $t\in[-\tau,0]$, and for $t\in[0,T]$ define
\[
y_p(t):=\int_0^t S_{\tau}^{a,b}(t-s)\,f(s)\,ds.
\]
Clearly $y_p\in C([-\tau,T])$ and $y_p(0)=0$. Since $\dot S_{\tau}^{a,b}$ is continuous on $[0,\infty)$, the map
$t\mapsto \dot S_{\tau}^{a,b}(t-s)f(s)$ is continuous for $t\ge 0$ and $s\in[0,t]$, and by Leibniz' rule we obtain for $t\in(0,T]$
\[
\dot y_p(t)=\int_0^t \dot S_{\tau}^{a,b}(t-s)\,f(s)\,ds,
\]
hence $\dot y_p\in C([0,T])$ and also $\dot y_p(0)=0$ because the integral has zero length at $t=0$.

Fix an integer $n\ge 0$ and let $t\in(n\tau,(n+1)\tau)\cap(0,T]$. On such an open step interval, the function $S_{\tau}^{a,b}$ is $C^2$ and
termwise differentiation of the defining series is legitimate; therefore we may differentiate $\dot y_p$ under the integral sign on this interval.
A second application of Leibniz' rule yields
\[
\ddot y_p(t)=\int_0^t \ddot S_{\tau}^{a,b}(t-s)\,f(s)\,ds+\dot S_{\tau}^{a,b}(0)\,f(t).
\]
From Theorem~\ref{combined1} we have $\dot S_{\tau}^{a,b}(0)=1$ and, for $t-s$ lying in an open step interval,
\[
\ddot S_{\tau}^{a,b}(t-s)=a\,S_{\tau}^{a,b}(t-s)+b\,S_{\tau}^{a,b}(t-s-\tau).
\]
Substituting this into the expression for $\ddot y_p(t)$ gives
\begin{align*}
\ddot y_p(t)
&=a\int_0^t S_{\tau}^{a,b}(t-s)\,f(s)\,ds
+b\int_0^t S_{\tau}^{a,b}(t-\tau-s)\,f(s)\,ds
+f(t)\\
&=a\,y_p(t)
+b\int_0^t S_{\tau}^{a,b}(t-\tau-s)\,f(s)\,ds
+f(t).
\end{align*}
Because $S_{\tau}^{a,b}(\eta)=0$ for $\eta<0$, we have for every $t\ge 0$
\[
\int_0^t S_{\tau}^{a,b}(t-\tau-s)\,f(s)\,ds=\int_0^{t-\tau} S_{\tau}^{a,b}(t-\tau-s)\,f(s)\,ds,
\]
with the convention that the right-hand side is $0$ if $t<\tau$. By the definition of $y_p$ this equals $y_p(t-\tau)$ (and for $t<\tau$ also
$y_p(t-\tau)=0$ because $y_p$ is zero on $[-\tau,0]$). Hence, for all $t\in(n\tau,(n+1)\tau)\cap(0,T]$,
\[
\ddot y_p(t)=a\,y_p(t)+b\,y_p(t-\tau)+f(t).
\]
Thus $y_p$ satisfies \eqref{eq98} pointwise on every open step interval, with $y_p(t)=0$ and $\dot y_p(t)=0$ on $[-\tau,0]$.
The possible jump discontinuities of $\ddot y_p$ at the grid points $t=k\tau$ come only from the corresponding jump behavior of
$\ddot S_{\tau}^{a,b}$, while $y_p$ and $\dot y_p$ remain continuous. Therefore $y_p$ is a stepwise classical solution of \eqref{eq98} with
zero initial history, and \eqref{13} holds. This completes the proof.
\qed

\end{document}